\documentclass[12pt]{article}

\usepackage{subfigure}

\usepackage{latexsym}
\usepackage{graphics}
\usepackage{amsmath}
\usepackage{xspace}
\usepackage{amssymb}
\usepackage{psfrag}
\usepackage{epsfig}
\usepackage{amsthm}
\usepackage{pst-all}

\usepackage{color}

\definecolor{Red}{rgb}{1,0,0}
\definecolor{Blue}{rgb}{0,0,1}
\definecolor{Olive}{rgb}{0.41,0.55,0.13}
\definecolor{Yarok}{rgb}{0,0.5,0}
\definecolor{Green}{rgb}{0,1,0}
\definecolor{MGreen}{rgb}{0,0.8,0}
\definecolor{DGreen}{rgb}{0,0.55,0}
\definecolor{Yellow}{rgb}{1,1,0}
\definecolor{Cyan}{rgb}{0,1,1}
\definecolor{Magenta}{rgb}{1,0,1}
\definecolor{Orange}{rgb}{1,.5,0}
\definecolor{Violet}{rgb}{.5,0,.5}
\definecolor{Purple}{rgb}{.75,0,.25}
\definecolor{Brown}{rgb}{.75,.5,.25}
\definecolor{Grey}{rgb}{.5,.5,.5}

\newcommand{\R}{\mathbb{R}}
\newcommand{\Z}{\mathbb{Z}}
\newcommand{\G}{\mathbb{G}}

\newcommand{\ignore}[1]{\relax}

\newtheorem{theorem}{Theorem}[section]

\newtheorem{proposition}[theorem]{Proposition}

\definecolor{Red}{rgb}{1,0,0}
\definecolor{Blue}{rgb}{0,0,1}
\definecolor{Olive}{rgb}{0.41,0.55,0.13}
\definecolor{Green}{rgb}{0,1,0}
\definecolor{MGreen}{rgb}{0,0.8,0}
\definecolor{DGreen}{rgb}{0,0.55,0}
\definecolor{Yellow}{rgb}{1,1,0}
\definecolor{Cyan}{rgb}{0,1,1}
\definecolor{Magenta}{rgb}{1,0,1}
\definecolor{Orange}{rgb}{1,.5,0}
\definecolor{Violet}{rgb}{.5,0,.5}
\definecolor{Purple}{rgb}{.75,0,.25}
\definecolor{Brown}{rgb}{.75,.5,.25}
\definecolor{Grey}{rgb}{.5,.5,.5}
\definecolor{Pink}{rgb}{1,0,1}
\definecolor{DBrown}{rgb}{.5,.34,.16}
\definecolor{Black}{rgb}{0,0,0}

\usepackage{float}

\author{
}

\author{
{\sf David Gamarnik}\thanks{MIT; e-mail: {\tt gamarnik@mit.edu}. Support from ONR Grant N00014-17-1-2790 is gratefully acknowledged}
}

\begin{document}

\title{Explicit construction of RIP matrices  is Ramsey-hard}
\date{\today}

\maketitle

\begin{abstract}
 Matrices $\Phi\in\R^{n\times p}$ satisfying the Restricted Isometry Property (RIP) are an important ingredient of the compressive sensing methods. While it is 
known that random matrices satisfy the RIP with high probability even for $n=\log^{O(1)}p$, 
the explicit deteministic construction of such matrices defied the repeated efforts, and  most of the known
approaches hit the so-called $\sqrt{n}$ sparsity bottleneck. The notable exception is the work by Bourgain et al \cite{bourgain2011explicit} constructing
an $n\times p$ RIP matrix with sparsity $s=\Theta(n^{{1\over 2}+\epsilon})$, but in the regime $n=\Omega(p^{1-\delta})$.

In this short note we  resolve this open question in a sense by showing that an explicit construction of a matrix satisfying the RIP in the regime
$n=O(\log^2 p)$ and $s=\Theta(n^{1\over 2})$
implies an explicit construction of a three-colored Ramsey graph on $p$ nodes with clique sizes bounded by $O(\log^2 p)$ --
a question in the field of extremal combinatorics which has been open for decades.
\end{abstract}


\section{Background}
An $n\times p$ real valued matrix $\Phi$ with unit norm column vectors is said to satisfy $(s,\delta)$ 
Restricted Isometry Property (RIP) for $s\le p$ and $\delta\in (0,1)$ if $|\|\Phi x\|_2^2-1|\le \delta$ for all $x\in\R^p, \|x\|_2=1$  which
are $s$-sparse, namely have at most $s$ non-zero coordinates. In this case we simply say that the matrix $\Phi$ is RIP.
The RIP is of fundamental importance in compressive sensing methods~\cite{foucart2013mathematical},\cite{buhlmann2011statistics}.
Provided $\Phi$ is RIP, the linear programming based method enable efficient (polynomial time) unique recovery of the solution $x$
of the linear system $y=\Phi x$, from $y$ and $\Phi$, whenever $x\in\R^p$ is $2s$-sparse. It is known that generating entries
of $\Phi$ i.i.d. from a common distribution, satisfying minor properties such as sub-Gaussianity and then normalizing columns of $\Phi$ to unit
norm, guarantees RIP provided $n=\Omega(s\log(p/s))$~\cite{foucart2013mathematical}. 
Here $O(\cdot),\Omega(\cdot)$ and $\Theta(\cdot)$ are standard order of magnitude notations.
This includes the case when $s=\log^{O(1)}p$. As a result it suffices to have  poly-log values of $n$ as well: $n=\log^{O(1)}p$. 

At the same time,  verifying whether a given matrix $\Phi$ satisfies the RIP is tricky, as the problem of certifying RIP of a matrix
in the worst case is NP-hard~\cite{bandeira2013certifying}. Furthermore,  even determining the RIP value $\delta$ up to a certain approximation factor
is hard in the average case sense, as shown in~\cite{koiran2014hidden} by reduction from the Planted Clique Problem. Motivated by this
complications, the researchers have sought to obtain explicit deterministic constructions of matrices satisfying the RIP. 
While there are constructions of RIP matrices when $s=O(\sqrt{n})$~\cite{applebaum2009chirp},
\cite{devore2007deterministic},\cite{fickus2010steiner},
most methods however break down when $s$ is at least some constant times $\sqrt{n}$, see Bandeira et al.~\cite{bandeira2013road} for a survey of the known methods,
leaving this case as an open problem. This problem  since then was popularly dubbed as the so-called ''square root bottleneck'' problem. 
It  was also raised  in 
Tao's blog~\cite{TaoUUPBlog}, as well as in a blog discussion by Moreira~\cite{MoreiraRIPBlog}. 
A notable exception to the square root bottleneck
is the construction by Bourgain et al. \cite{bourgain2011explicit} which breaks this barrier by achieving $s=\Theta(n^{{1\over 2}+\epsilon})$
for some small constant $\epsilon>0$, in the regime $n=\Omega(p^{1-\delta})$ and with matrix containing complex valued entries. Thus the
cases of poly-logarithmic $n$ and $s$,  or even polynomial $n$ and $s$ with real valued entries remain open.

In this short note we give a very simple argument that this problem is ''Ramsey-hard'' so to speak, see~\cite{graham1990ramsey} for a book
reference on Ramsey theory and~\cite{conlon2015recent} for a survey. Specifically, we show that an explicit construction
of an RIP matrix when $n=O(\log^2 p)$ and $s\ge 2\sqrt{n}+1$ implies an explicit construction of a 3-colored Ramsey graph on $p$
nodes with monochromatic clique sizes bounded by $O(\log^2 p)$. To elaborate,
a complete graph $K_p$ on $p$ nodes with $p(p-1)/2$ edges colored using $q$ colors is called $R(m_1,\ldots,m_q)$ Ramsey if the largest monochromatic clique
with respect to color $r$ (a subset of nodes with all induced edges colored $r$)
is at most $m_r$ for all $1\le r\le q$. When all $m_r$ are identically $m$ we simply write $R(m;q)$. 
While the uniform random coloring provides an easy construction of  $R(m;q)$ Ramsey graphs with $m=O(\log p)$ w.h.p. as $p\to\infty$, explicit construction
of such graphs is a major open problem in the extremal combinatorics even for larger order of magnitude $m$. The best known construction
for $q=2$ gives $m=\left(\log p\right)^{\log\log\log^{O(1)}p}$~\cite{cohen2017towards}. The case $m=\log^{O(1)} p$
is believed to be out of reach using the known methods, and the problem of explicit construction of graphs satisfying this Ramsey type property 
has been open since the celebrated work of Erd{\"o}s~\cite{erdos1947some}.

\section{Construction}
Suppose $\Phi\in \R^{n\times p}$ is a matrix with unit norm column vectors, $\Phi=(u_i, 1\le i\le p)$, $\|u_i\|_2=1$,
which satisfies the $(s,\delta)$-RIP with $\delta\in (0,1)$ and $s\ge 2\sqrt{n}+1$. 
Consider the complete graph $K_{[p]}$ 
on the set of nodes $[p]=\{1,2,\ldots,p\}$. We color all of the ${p\choose 2}$ edges of this graph 
with three colors, white, blue and red, as follows. For each $1\le i\ne j\le p$, the color 
of $(i,j)$ is white if $|\langle u_i,u_j\rangle|\le 1/(2\sqrt{n})$, blue if $\langle u_i,u_j\rangle > 1/(2\sqrt{n})$ and
red if $\langle u_i,u_j\rangle < -1/(2\sqrt{n})$. Here $\langle \cdot,\cdot\rangle$ denotes standard inner product in $\R^n$.

\begin{theorem}\label{theorem:main}
The graph $\G$ is $R(2n,2\sqrt{n}+1,2\sqrt{n}+1;3)$ Ramsey.
\end{theorem}

\begin{proof}

The following proposition is a simplified version of Kabatjanskii-Levenstein bound~\cite{kabatiansky1978bounds} discussed in 
Tao's blog \cite{TaoBlogKabatianskiiLewenstein}. We reproduce the proof from there for completeness.
\begin{proposition}\label{prop:orthogonality}
For any set of unit norm vectors $u_1,\ldots,u_{2n}\in \R^n$, $\max_{1\le i\ne j\le 2n}|\langle u_i,u_j\rangle| >{1\over 2\sqrt{n}}$.
\end{proposition}

\begin{proof}
Consider the symmetric matrix $U=(\langle u_i,u_j\rangle, 1\le i,j\le 2n)\in\R^{2n\times 2n}$ of inner products. This is a rank-$n$ matrix in $\R^{2n\times 2n}$
and as such $\bar U\triangleq U-I_{2n\times 2n}$ has an eigenvalue $-1$ with multiplicity at least $n$. Thus the trace of $\bar U^2$
which is $\sum_{1\le i\ne j\le 2n}(\langle u_i,u_j\rangle )^2$ is at least $n$, implying  $\max_{i\ne j}|\langle u_i,u_j\rangle|\ge {1\over \sqrt{2(2n-1)}}$.
\end{proof}

Now suppose
$C\subset [p]$ is a white clique. Then by Proposition~\ref{prop:orthogonality}, $|C|<2n$. Suppose $C\subset [p]$ is a blue clique
which for the purposes of contradiction satisfies $|C|\ge 2\sqrt{n}+1$.
Take any subset of $C$ with cardinality $2\sqrt{n}+1$. For simplicity denote it by $C$ as well. 
Let $x\in \R^p$ be the unit norm $\|x\|_2=1$ vector defined by $x_i=1/\sqrt{|C|}, i\in C$ and $x_i=0$ otherwise. This vector is $|C|\le s$-sparse.
Since $C$ supports a blue clique then
\begin{align*}
\|\Phi x\|_2^2-\|x\|_2^2={1\over |C|}\sum_{i\ne j\in C}\langle u_i,u_j\rangle \ge {|C|-1\over 2\sqrt{n}}=1,
\end{align*}
which contradicts RIP. The same argument applies for red cliques. We conclude that our 
graph is  $R(2n,2\sqrt{n}+1,2\sqrt{n}+1;3)$ Ramsey,  and therefore crudely $R(2n;3)$ Ramsey.

\end{proof}

Now find $C>0$ and $\delta\in (0,1)$ so that for $n\ge Cs\log p$ the random $n\times p$ matrix $\Phi$ is RIP w.h.p. for all large enough $p$
for sparsity $s$ and parameter $\delta$.
Set $n=9C^2\log^2 p$ and $s=9C\log p=3\sqrt{n}>2\sqrt{n}+1$. Then $n=Cs\log p$. 
The explicit construction of RIP matrix $\Phi\in\R^{n\times p}$ for this choice of $n$,  
sparsity $s$, and parameter $\delta$  implies, by the above, an explicit construction of an $R(2n;3)$ Ramsey graph on $p$
nodes with $n=9C^2\log^2p=O(\log^2 p)$ -- the task which appears not yet reachable with known techniques. 

We note that if the entries of the matrix $\Phi$ are assumed to be non-negative, and in fact some known explicit constructions are based
on non-negative entries~\cite{bandeira2013road} (see below for one such construction), then we can restrict our construction above to just two colors,
corresponding to the cases $0\le \langle u_i,u_j\rangle \le 1/(2\sqrt{n})$ and $\langle u_i,u_j\rangle > 1/(2\sqrt{n})$. Thus in this
case the explicit construction of a matrix $\Phi$ satisfying the stated RIP implies an explicit construction  of a $R(2n;2)$ graph -
an explicit $2$-coloring of edges of a complete graph on $p$ nodes with largest monochromatic cliques of sizes at most $O(\log^2 p)$.

The  case $n=p^{O(1)}$, covered
by~\cite{bourgain2011explicit} remains open for general $s=n^{{1\over 2}+\epsilon}$. 
It would be also interesting to see if, conversely, the  Ramsey graphs can be used to construct the RIP matrices.
It also would be interesting to see if the construction above can be extended to the complex valued matrices in the poly-logarithmic case.

We close this note by suggesting one additional open question. While the case of explicit constructions is solved in the regime $s\ge c\sqrt{n}$
for sufficiently small constant $c>0$, all of the known constructions are in the polynomial regime $n=p^{O(1)},s=p^{O(1)}$. This raises a potential
issue as to whether the square root is even achievable in the polylogarithmic regime $n=\log^{O(1)} p$, or alternatively, in this regime the problem
is even more challenging. We now show that the answer is \emph{essentially} no, but closing this gap entirely is an open question.
 The following construction due to deVore~\cite{devore2007deterministic} leads to an example 
 satisfying $s\ge n^{{1\over 2}-\epsilon}$ for every $\epsilon>0$ in the polylogarithmic regime $n=\log^{O(1)} p$. Fix a prime number $z$ and a positive
 integer $r$. Let $p=z^{r+1}$ be the number of all degree $\le r$ polynomials with coefficients in $\Z_z=\{0,1,\ldots,z-1\}$.
 Consider the $z^2\times p$ matrix $\Phi$ constructed as follows. The rows are indexed by pairs $(x,y)\in\Z_z^2$. Thus 
 the number of rows is $n=z^2$.
 The columns
 are indexed by polynomials described above. For every such polynomial $P$ and every pair $(x,y)$, the matrix entry corresponding 
 to the location $((x,y),P)$ is set to $1/\sqrt{z}$ if $y=P(x)$ and $=0$ otherwise. It is clear that the columns $u_j, 1\le j\le p$
 of this matrix are unit length $\|u\|_2=1$ as they contain precisely $z$ non-zero entries, each valued at $1/\sqrt{z}$. For every two
 distinct columns $u_i,u_j, i\ne j$ we have $|\langle u_i,u_j\rangle|\le r/z$. Indeed, the row $(x,y)$ has non-zero (namely $1/\sqrt{z}$) entries
 in columns indexed by polynomials $P$ and $Q$ only if $y=P(x)=Q(x)$, that is $x$ is a root of $P-Q$, of which only $r$ exist. 
 Thus the coherence of this matrix, defined as $\max_{i\ne j}|\langle u_i,u_j\rangle|$ is at most $r/z$. As a result, the matrix satisfies the $(s,\delta)$
 RIP when $sr/z\le \delta$ (see~\cite{devore2007deterministic} for details). 
 
 Now we fix arbitrary $\epsilon>0$. For each sufficiently large $z$ we set $r=z^\epsilon$ and consider the $z^2\times p$ matrix described
 above with $p=z^{r+1}$. We have $n=z^2$ and the sparsity level satisfying the RIP can be taken as 
 \begin{align*}
 s=(1/2)(z/r)=(1/2)z^{1-\epsilon}=(1/2)n^{{1\over 2}-{\epsilon\over 2}}.
\end{align*} 
Namely, the construction almost achieves the square root barrier. At the same time 
\begin{align*}
p=z^{r+1}>z^r=z^{z^\epsilon}=\exp(z^\epsilon\log z),
\end{align*}
implying
\begin{align*}
n=z^2=\left({\log p\over \log z}\right)^{2\over \epsilon}\le \left(\log p\right)^{2\over \epsilon}.
\end{align*}
Namely, the sample size $n$ is indeed poly-logarithmic in the number of columns $p$. It would be interesting to obtain a construction which
truly achieves the square root barrier $s\ge C\sqrt{n}$ in the polylogarithmic regime $n=\log^{O(1)} p$.

\section*{Acknowledgement} The author thanks Gil Cohen, Jacob Fox and Benny Sudakov for communicating the state of the art results 
on the explicit construction of Ramsey graphs, and Afonso Bandeira for many discussions about constructing RIP matrices. Discussions
with Joel Moreira are also gratefully acknowledged.

\bibliographystyle{amsalpha}

\newcommand{\etalchar}[1]{$^{#1}$}
\providecommand{\bysame}{\leavevmode\hbox to3em{\hrulefill}\thinspace}
\providecommand{\MR}{\relax\ifhmode\unskip\space\fi MR }
\providecommand{\MRhref}[2]{%
  \href{http://www.ams.org/mathscinet-getitem?mr=#1}{#2}
}
\providecommand{\href}[2]{#2}

\end{document}